\documentclass{aic}

\aicAUTHORdetails{%
  title = {Big Ramsey Degrees and Infinite Languages}, 
  author = {Samuel Braunfeld, David Chodounsk\'y, No\'e de Rancourt, Jan Hubi\v cka, Jamal Kawach, and Mat\v ej Kone\v{c}n\'{y}},
  plaintextauthor = {Samuel Braunfeld, David Chodounsky, Noe de Rancourt, Jan Hubicka, Jamal Kawach, Matej Konecny},
  runningauthor = {S. Braunfeld, D. Chodounsk\'y, N. de Rancourt, J. Hubi\v cka, J. Kawach, and M. Kone\v{c}n\'{y}},
}   

\aicEDITORdetails{%
   year={2024},
   number={4},
   received={10 July 2023},   
   published={10 August 2024},  
   doi={10.19086/aic.2024.4},      
}   

\def\Succ{\mathop{\mathrm{Succ}}\nolimits}
\def\ImmSucc{\mathop{\mathrm{IS}}\nolimits}
\def\Str{\mathop{\mathrm{Str}}\nolimits}

\def\branch#1#2{\mathrm{branch}^{#1}(#2)}
\def\nbrel#1#2{R\ifstrempty{#1}{}{_{#1}}\ifstrempty{#2}{}{^{#2}}}
\def\rel#1#2{\nbrel{\ifstrempty{#1}{}{\str{#1}}}{#2}}
\def\Sym{\mathop{\mathrm{Sym}}\nolimits}
\def\arity{\mathop{\mathrm{a}}\nolimits}

\newcommand{\conc}{^\smallfrown}           
\newcommand{\card}[1]{\left| #1 \right|}    
\DeclareMathOperator{\val}{val}       

\def\tp{\mathop{\mathrm{tp}}\nolimits}

\theoremstyle{definition}
\newtheorem{definition}{Definition}[section]
\newtheorem{example}{Example}
\newtheorem*{remark*}{Remark}
\newtheorem*{claim*}{Claim}
\theoremstyle{remark}
\newtheorem{remark}{Remark}[section]

\theoremstyle{plain}
\newtheorem{theorem}{Theorem}[section]
\newtheorem{corollary}[theorem]{Corollary}
\newtheorem{prop}[theorem]{Proposition}
\newtheorem{observation}[theorem]{Observation}
\newtheorem{lemma}[theorem]{Lemma}
\newtheorem{conjecture}[theorem]{Conjecture}
\newtheorem{question}[theorem]{Question}
\newtheorem{problem}[theorem]{Problem}

\newtheorem{claim}[theorem]{Claim}

\def\str#1{\mathbf{#1}}
\def\Fraisse{Fra\"{\i}ss\' e}

\begin{document}

\begin{frontmatter}[classification=text]



\author[sam]{Samuel Braunfeld\thanks{Supported by a project that has received funding from the European Research Council (ERC) under the European Union’s Horizon 2020 research and innovation programme (grant agreement No 810115) and by project 21-10775S of the Czech Science Foundation (GA\v CR).}}
\author[david]{David Chodounsk\'y\thanks{Supported by project 21-10775S of the Czech Science Foundation (GA\v CR).}}
\author[noe]{No\'e de Rancourt\thanks{Acknowledges support from the Labex CEMPI (ANR-11-LABX-0007-01).}}
\author[honza]{Jan Hubi\v cka\thanks{Supported by a project that has received funding from the European Research Council (ERC) under the European Union’s Horizon 2020 research and innovation programme (grant agreement No 810115) and by project 21-10775S of the Czech Science Foundation (GA\v CR).}}
\author[jamal]{Jamal Kawach\thanks{Supported by project 21-10775S of the Czech Science Foundation (GA\v CR).}}
\author[matej]{Mat\v ej Kone\v{c}n\'{y}\thanks{Supported by project 21-10775S of the Czech Science Foundation (GA\v CR) in the earlier stages of the project. In the later stages of the project supported by the European Research Council (Project POCOCOP, ERC Synergy Grant 101071674). Views and opinions expressed are however those of the authors only and do not necessarily reflect those of the European Union or the European Research Council Executive Agency. Neither the European Union nor the granting authority can be held responsible for them.}}

\begin{abstract}
This paper investigates big Ramsey degrees of unrestricted relational structures in (possibly) infinite languages. Despite significant progress in the study of big Ramsey degrees, the big Ramsey degrees of many classes of structures with finite small Ramsey degrees are still not well understood. We show that if there are only finitely many relations of every arity greater than one, then unrestricted relational structures have finite big Ramsey degrees, and give some evidence that this is tight. This is the first time finiteness of big Ramsey degrees has been established for a random structure in an infinite language. Our results represent an important step towards a better understanding of big Ramsey degrees for structures with relations of arity greater than two.
\end{abstract}
\end{frontmatter}

\section{Introduction}

Given a language $L$ and $L$-structures $\str{A}$, $\str B$ and $\str{C}$, we denote by ${\str B\choose \str A}$ the set of all embeddings $\str A\to \str B$, and we write
$\str{C}\longrightarrow {(\str{B})}^\str{A}_{k,\ell}$
to denote the following statement:
\begin{quote}
  For every colouring
  $\chi\colon{\str{C}\choose\str{A}}\to k$ there exists an embedding
  $f\colon\str{B}\to\str{C}$ such that $\chi$ takes at most $\ell$ values on
  $f[\str{B}]\choose \str{A}$.
\end{quote}
For a countably infinite structure~$\str{B}$ and its finite substruc\-ture~$\str{A}$,
the \emph{big Ramsey degree} of $\str{A}$ in $\str{B}$ is the least number
$\ell\in \mathbb \omega + 1$ 
such that $\str{B}\longrightarrow
  {(\str{B})}^\str{A}_{k,\ell}$ for every $k\in \mathbb \omega$; see~\cite{Kechris2005}.
A countably
infinite structure $\str{B}$ has \emph{finite big Ramsey degrees} if the big
Ramsey degree of $\str{A}$ in $\str{B}$ is finite for every finite substructure $\str{A}$ of $\str{B}$.

The study of big Ramsey degrees dates back to Ramsey's theorem itself which can be stated as $(\omega)\longrightarrow (\omega)^n_{k,1}$ for every $n,k\in \omega$, where we understand the ordinals $\omega$ and $n$ as structures with their natural linear orders (using the standard set-theoretic convention that $n = \{0,1,\ldots,n-1\}$ and $\omega = \{0, 1, \ldots\}$). However, the real origin of this project lies in the work of Galvin~\cite{Galvin68,Galvin69} who proved that 
the big Ramsey degree of pairs in the order of the rationals, denoted by $(\mathbb Q,\leq)$, is equal to 2, building on a lower bound by Sierpi\'nski~\cite{Sierpinski1933}. 
Subsequently, Laver in late 1969 proved that $(\mathbb Q,\leq)$ has in fact finite big Ramsey degrees, 
see~\cite{erdos1974unsolved,laver1984products}, and 
Devlin determined the exact values of $\ell$~\cite[Page~73]{devlin1979}. In 2006 Sauer~\cite{Sauer2006} proved that the Rado graph has finite big Ramsey degrees and Laflamme, Sauer, and Vuksanovic~\cite{Laflamme2006} obtained their exact values. Behind both this result and the result for $(\mathbb Q,\leq)$ was Milliken's tree theorem for a single binary tree. This was refined to unconstrained
structures in binary languages~\cite{Laflamme2006} and additional special
classes~\cite{dobrinen2016rainbow,laflamme2010partition,NVT2009}. Until very recently, Milliken's tree theorem has remained
the key ingredient in all results in the area (note that Ramsey's theorem is a special case of Milliken's tree theorem for the unary tree).
See also~\cite{dobrinen2021ramsey} for a recent survey.

Recently, Balko, Chodounsk{\'y}, Hubi{\v{c}}ka, Kone{\v{c}}n{\'y}, and Vena~\cite{Hubickabigramsey, Hubicka2020uniform} applied the vector version of Milliken's tree theorem to prove that the generic countable 3-uniform hypergraph has finite big Ramsey degrees. The method can be extended to prove finiteness of big Ramsey degrees of the generic countable $k$-uniform hypergraph for an arbitrary $k$, and in this paper we further extend these results and prove the following theorem (the definition of an unrestricted structure is given later, see Definition~\ref{def:unrestricted}):

\begin{theorem}\label{thm:structures}
Let $L$ be a relational language with finitely many relations of every arity greater than one and with finitely or countably many unary relations and let $\str H$ be a countable unrestricted $L$-structure where all relations are injective. Then $\str H$ has finite big Ramsey degrees.
\end{theorem}
Note that one has to require that all relations are injective as otherwise, by repeating vertices, one can use higher-arity relations as lower-arity and thus effectively have infinitely many binary relations even in this context.

We believe that this result is the limit of how far Milliken's tree theorem can be pushed in this area (at least by using the \emph{passing number representation} and its generalisations). In Section~\ref{sec:inf} we give evidence for this and discuss infinite lower bounds. In fact, we conjecture that this theorem is tight when there are only finitely many unary relations (see Section~\ref{subsec:inf} and Conjecture~\ref{conj:inf}).

Besides Milliken's tree theorem, other partition theorems have been used in the area, such as the Carlson--Simpson theorem~\cite{Hubickabigramsey2,balko2021big,Hubicka2020CS}, various custom theorems proved using forcing~\cite{coulson2022SDAP,dobrinen2017universal,dobrinen2019ramsey,zucker2020}, as well as some recent non-forcing custom tree theorems~\cite{Balko2023,Balko2023Sucessor}. While we only use Milliken's theorem here, the vector tree structure we develop is, to a large degree, inherent to the problem, not to the method. For this reason we believe that our development of the concept of valuation trees (extending~\cite{Hubicka2020uniform}) and $k$-enveloping embeddings will serve as an important basis for future big Ramsey theorems for structures with relations of arity greater than two.

\section{Preliminaries}

A \emph{relational language} $L$ is a collection of symbols, each having an associated \emph{arity}, denoted by $\arity(\rel{}{}) \in \omega$. An \emph{$L$-structure} $\str A$ consists of a \emph{vertex set} $A$ and an \emph{interpretation} of every $\rel{}{}\in L$, which is $\rel{A}{} \subseteq A^{\arity(\rel{}{})}$.
We say that a relation $\rel{}{}$ is \emph{injective} in $\str A$ if every tuple $\bar{x}\in \rel{A}{}$ is injective (i.e. contains no repeated occurrences of vertices) and it is \emph{symmetric} if whenever $\bar{x}\in \rel{A}{}$ and $\bar{y}$ is a permutation of $\bar{x}$ then $\bar{y}\in \rel{A}{}$. Equivalently, we can consider a symmetric relation to be a subset of ${A\choose {\arity(\rel{}{})}}$. An \emph{$L$-hypergraph} is an $L$-structure where all relations are injective and symmetric and every tuple is in at most one relation (so it can be seen as an edge-colored hypergraph with the number of colours for each arity given by $L$).

We adopt the standard model-theoretic notions of embeddings etc. with one exception: Unless explicitly stated otherwise, every structure in this paper will be implicitly equipped with an enumeration (i.e. a linear order which has type $\omega$ for countably infinite structures) and all embeddings will be monotone with respect to the enumerations. When we do not explicitly describe the enumeration, one can pick an arbitrary one. This is common in the area, as (for structures with no algebraicity) working with enumerated structures preserves whether they have finite big Ramsey degrees or not (see e.g. Section~1 of~\cite{zucker2020}).

Since all structures will be countable (that is, finite or countably infinite), we can assume without loss of generality that the vertex set of every structure is either some natural number $n$ or $\omega$, the symbol $\leq$ will always denote the standard order of natural numbers, and all embeddings will be monotone with respect to $\leq$. We will always assume that $\leq\notin L$. If $\str A$ and $\str B$ are $L$-structures, the symbol $\str B\choose \str A$ denotes the set of all embeddings $\str A\to \str B$. (Remember that these are monotone with respect to $\leq$.)

Given a class of finite and countably infinite structures $\mathcal C$, a structure $\str A$ is \emph{universal for $\mathcal C$} if for every $\str B\in \mathcal C$ there exists an embedding $\str B\to \str A$. Examples of universal structures come, for example, from \Fraisse{} theory which, under some conditions on $\mathcal C$ (e.g. strong amalgamation) produces special (homogeneous) unenumerated structures which retain their universality for all enumerations of members of the respective class even when enumerated.

\begin{definition}\label{def:unrestricted}
An $L$-structure $\str F$ is \emph{covered by a relation} if there is some relation $\rel{}{}\in L$ and a tuple $\bar x$ containing all vertices of $\str F$ such that $\bar x\in \rel{F}{}$. If $\mathcal F$ is a collection of $L$-structures and $\str A$ is an $L$-structure such that there is no $\str F\in \mathcal F$ with an embedding $\str F\to \str A$, we say that $\str A$ is \emph{$\mathcal F$-free}. We say and an $L$-structure $\str A$ is \emph{unrestricted} if there is a family $\mathcal F$ containing only finite $L$-structures which are covered by a relation such that $\str A$ is $\mathcal F$-free and universal for all countable $\mathcal F$-free structures.
\end{definition}

\subsection{Milliken's tree theorem}

Our argument will make use of the vector form of Milliken's tree theorem. In order to formulate it, we need to give some standard definitions, see e.g.~\cite{todorcevic2010introduction}.
Given an integer~$\ell$, we use the set-theoretical convention $\ell = \{ 0,1,\ldots,\ell-1 \}$.

A \emph{tree} is a (possibly empty) partially ordered set $(T, <_T)$ such
that, for every $t \in T$, the set $\{ s \in T : s <_T t \}$ is finite and linearly ordered by $<_T$.
All trees considered in this paper will be countable.
All nonempty trees we consider are \emph{rooted}, that is, they have a unique minimal element called the \emph{root} of the tree.
An element $t\in T$ of a tree $T$ is called a \emph{node} of $T$ and its \emph{level},
denoted by $\card{t}_T$, is the size of the set $\{s \in T : s <_T t\}$.
Note that the root has level 0.
For $D \subseteq T$, we write $L_T(D) = \{ \card{t}_T : t \in D \}$ for the \emph{level set} of $D$ in $T$.
We use $T(n)$ to denote the set of all nodes of $T$ at level $n$,
and by $T({<}n)$ the set $\{ t\in T \colon \card{t}_T < n \}$.
The \emph{height} of $T$ is the smallest natural number $h$ such that $T(h)=\emptyset$.
If there is no such number $h$, then we say that the height of $T$ is $\omega$.
We denote the height of $T$ by $h(T)$.

Given a tree $T$ and nodes $s, t \in T$ we say that $s$ is a \emph{successor} of $t$ in $T$ if $t \leq_T s$.
The node $s$ is an \emph{immediate successor} of $t$ in $T$  if
$t<_T s$ and there is no $s'\in T$ such that $t<_T s'<_T s$.
We denote the set of all successors of $t$ in $T$ by $\Succ_T (t)$
and the set of immediate successors of $t$ in $T$ by $\ImmSucc_T (t)$.
We say that the tree $T$ is \emph{finitely branching} if $\ImmSucc_T (t)$ is finite for every $t \in T$.

For $s, t \in T$, the \emph{meet} of $s$ and $t$, denoted by $s\wedge_T t$, is the largest $s' \in T$ such that $s' \leq_T s$ and $s' \leq_T t$.
A node $t\in T$ is \emph{maximal} in $T$ if it has no successors in $T$.
The tree $T$ is \emph{balanced} if it either has infinite height and no maximal nodes, or all its maximal nodes are in
$T(h-1)$, where $h$ is the height of $T$. Given $t\in T$ and a level $\ell$ such that $\ell \leq \lvert t\rvert$, we denote by $t|_\ell$ the (unique) predecessor of $t$ on level $\ell$ and call it the \emph{restriction} of $t$ to level $\ell$.

A \emph{subtree} of a tree $T$ is a subset $T'$ of $T$ viewed as a tree
equipped with the induced partial ordering
such that $s \wedge_{T'} t = s \wedge_{T} t$ for each $s,t \in T'$.

\begin{definition}
  A subtree $S$ of a tree $T$ is a \emph{strong subtree} of $T$
  if either $S$ is empty, or $S$ is nonempty and satisfies the following three conditions.
  \begin{enumerate}
    \item The tree $S$ is rooted and balanced.
    \item Every level of $S$ is a subset of some level of $T$, that is, for every $n < h(S)$ there exists $m \in \omega$ such that $S(n) \subseteq T (m)$.
    \item For every non-maximal node $s \in S$ and every $t \in \ImmSucc_T (s)$ the set
          $\ImmSucc_S (s) \cap \Succ_T (t)$ is a singleton.
  \end{enumerate}
\end{definition}

\begin{observation}\label{obs-subtree}
  If $E$ is a subtree of a balanced tree $T$, then there exists a strong subtree  $S \supseteq E$ of $T$ such that $L_T(E) = L_T(S)$.\qed
\end{observation}

A \emph{vector tree} of dimension $d\in\omega+1$ (often also called \emph{product tree}) is a sequence $\mathbf T = (T_i : i\in d)$ of trees having the same height $h(T_i)$ for all $i \in d$.
This common height is the \emph{height} of $\mathbf T$
and is denoted by $h(\mathbf T)$. A vector tree $\mathbf T = (T_i : i\in d)$ is \emph{balanced}
if the tree $T_i$ is balanced for every $i \in d$.

If $\mathbf T = (T_i : i\in d)$ is a vector tree, then a \emph{vector subset} of $\mathbf T$ is a sequence
$\mathbf D = (D_i : i\in d)$ such that $D_i \subseteq T_i$ for every $i \in d$.
We say that $\mathbf D$ is \emph{level compatible} if there exists $L \subseteq \omega$ such
that $L_{T_i} (D_i ) = L$ for every $i \in d$.
This (unique) set $L$ is denoted by $L_\mathbf T(\mathbf D)$
and is called the \emph{level set} of $\mathbf D$ in $\mathbf T$. If $k\in d$, we denote $\mathbf T\restriction_k = (T_i : i\in k)$.

\begin{definition}
\label{def:T1T2}
  Let $\mathbf T = (T_i : i\in d)$ be a vector tree.
  A \emph{strong vector subtree} of~$\mathbf T$ is a level compatible vector subset $\mathbf S = (S_i : i\in d)$ of $\mathbf T$
  such that $S_i$ is a strong subtree of $T_i$ for every $i \in d$.
\end{definition}

For every $k \in \omega+1$ with $k \leq h(\mathbf T)$, we use $\Str_k (\mathbf T)$ to denote the set of all strong vector
subtrees of $\mathbf T$ of height $k$.
We also use $\Str_{\leq k}(\mathbf T )$ to
denote the set of all strong vector subtrees of $\mathbf T$ of height at most $k$.

\begin{theorem}[Milliken~\cite{Milliken1979}]\label{thm:Milliken}
  For every rooted, balanced and finitely branching vector tree $\mathbf T$ of infinite height and finite dimension,
  every non-negative integer $k$ and every finite colouring of $\Str_k (\mathbf T)$ there is $\mathbf S \in \Str_\omega(\mathbf T)$
  such that the set $\Str_k (\mathbf S)$ is monochromatic.
\end{theorem}

\section{Valuation trees}

Given $n\in \omega+1$ and $0 < \ell < \omega$, we denote $I^n_{\ell} = \{(i_0,...,i_{\ell-1}) : n > i_0  > \cdots > i_{\ell-1} \geq 0\}$. If $n,\ell\in \omega+1$, we put $I^n_{<\ell} = \bigcup_{1\leq k < \ell} I^n_{k}$. Let $\sigma = (\sigma_1, \sigma_2,\ldots)$ be an infinite sequence of positive natural numbers, we will call it a \emph{signature}. Let $f\colon I^n_{<\omega} \to \omega$ be a function such that if $\bar{x}\in I^n_\ell$ then $f(\bar{x}) < \sigma_\ell$. We call such $f$ a \emph{valuation function of level $n$ and signature $\sigma$} (or a \emph{$\sigma$-valuation function of level $n$}, or just a \emph{valuation function of level $n$} if $\sigma$ is clear from the context), and write $\lvert f\rvert = n$. Given $i\in \omega$ we denote by $\sigma^{(i)}$ the \emph{$i$-shift of $\sigma$} defined by $\sigma^{(i)}_j = \sigma_{j+i}$.

Given a signature $\sigma$, let $\str T^\sigma = (T_0, T_1, \ldots)$ be the infinite-dimensional vector tree where $T_i$ consists of all $\sigma^{(i)}$-valuation functions of a finite level ordered by inclusion. Often $\sigma$ will be implicit from the context and we will write only $\str T$ for $\str T^\sigma$. Note that if $\sigma$ is constant 1 from some point on, $\str T^\sigma$ is essentially finite-dimensional (in the sense that from some point on, each $T_i$ is just the chain of constant zero valuation functions).

\medskip

Note that $T_0$ corresponds to the tree of all quantifier-free types of an $L$-hypergraph for $L$ having $\sigma_i - 1$ relations of arity $i+1$ and no unary relations: For example, let $\str M$ be a countably infinite graph (remember that it is enumerated with vertex set $\omega$) and fix some $n\in \omega$. A quantifier-free type of a vertex $v\in M$ over $n$ such that $v\notin n$ corresponds to just deciding which vertices from $n$ are neighbours of $v$, or in other words, we can represent it by a function $n\to 2$. Analogously, if $\str M$ is a countable 3-uniform hypergraph, quantifier-free types correspond to functions $I^n_2\to 2$.

In Section~\ref{sec:lhypergraphs} we will use $\str T^\sigma$ to prove finiteness of big Ramsey degrees for $L$-hypergraphs, and in the proof $T_0$ will carry a hypergraph structure, while the other trees $T_i$, $i>1$ will be auxiliary objects for the run of Milliken's theorem. Formally, the structure will be defined as follows (note that for technical reasons we are going to need two values to both represent non-relations).

\begin{definition}[An $L$-hypergraph on $T_0$]\label{defn:t0_structure}
Given a signature $\sigma$, define a language $L = \{ \rel{}{i,j} : 2\leq i < \omega, 1\leq j \leq \sigma_{i-1}-2 \}$ such that $\rel{}{i,j}$ has arity $i$ (note that if $\sigma_{i-1}\leq 2$ then there are no relations of arity $i$). For $\str T^\sigma = (T_0, T_1, \ldots)$, we define an $L$-hypergraph $\str G^\sigma$ with vertex set $T_0$ such that for every $i\geq 2$, every $1\leq j \leq \sigma_{i-1}-2$, and every $x_0,\ldots, x_{i-1}\in T_0$ with $\lvert x_0\rvert > \cdots > \lvert x_{i-1}\rvert$, we put
$$\{x_0, \ldots, x_{i-1}\} \in \nbrel{\str G^\sigma}{i,j} \iff x_0(\lvert x_1\rvert,\ldots, \lvert x_{i-1}\rvert) = j.$$

Equip $\str G^\sigma$ with the enumeration defined by $f\leq g$ if and only if either $\lvert f\rvert < \lvert g\rvert$, or $\lvert f\rvert = \lvert g\rvert$ and $f(\bar x) < g(\bar x)$, where $\bar x$ is the lexicographically smallest tuple where $f$ and $g$ differ.
\end{definition}

\medskip

A key element of our construction is a correspondence between strong vector subtrees of the vector tree $\str{T}^\sigma$ and the so-called valuation trees which we now define.
\begin{definition}
If $f$ is a $\sigma^{(i)}$-valuation function of level $n$, $g$ is a $\sigma^{(i+1)}$-valuation function of level $n$ and $h$ is a $\sigma^{(i)}$-valuation function of level $n+1$, we say that $h$ is an \emph{extension} of $f$ by $g$ if
$$h(x_0,\ldots,x_d) = \begin{cases}
f(x_0,\ldots,x_d) &\text{ if } x_0 < n,\\
g(x_1,\ldots,x_d) &\text{ if } x_0 = n\text{ and } d>0.
\end{cases}$$
We put $$f\conc g = \{h : h \text{ is an extension of $f$ by $g$}\}.$$
Note that $h\in \ImmSucc_{T_i}(f)$ and that there are always $\sigma^{(i)}_1$-many possible extensions of $f$ by $g$ which differ on their value at the singleton $n$.
\end{definition}

\begin{definition}\label{def:valuation_tree}
  Let $\sigma$ be an arbitrary signature, consider $\str T = \str T^\sigma$, let $k\in \omega + 1$ and let $(S_i : i\in k)$ be a strong vector subtree of $\str T\restriction_k$ of height at least $k$. We now define by induction on $k$ a subset of $S_0$ which we call the \emph{valuation tree $\val(S_i : i\in k)$}:

  If $k = 0$ then $\val() = \emptyset$. For $0 < k < \omega$, we put $S' = \val(S_1,\ldots,S_{k-1})$ using the induction hypothesis and define $\val(S_i : i\in k)$ by the following recursive rules:
  \begin{enumerate}
    \item The root of $\val(S_i : i\in k)$ is the root of $S_0$.
    \item If $f \in \val(S_i : i\in k)$, $g \in S'(\card{f}_{S_0})$, then $\ImmSucc_{S_0}(f) \cap \Succ_{T_0}(f\conc g) \subseteq \val(S_i : i\in k)$ (note that this set has size $\sigma^{(i)}_1$).
    \item There are no other nodes in $\val(S_i : i\in k)$.
  \end{enumerate}
  If $k=\omega$ then we put $\val(S_i : i\in k) = \bigcup_{j\in \omega} \val(S_i : i\in j)$. A tree $T \subseteq T_0$ is a \emph{valuation tree} if $T = \val(S_i : i\in k)$ for some $(S_i : i\in k) \in \Str_{\leq\omega}(\str{T})$.
\end{definition}

\begin{observation}\leavevmode
\begin{enumerate}
  \item The set $\val(S_i : i\in k)$ is a subtree of $S_0$, and hence also a subtree of~$T_0$, of height $k$.
  \item If $(S_i : i\in k)$ is a strong vector subtree of $\str T\restriction_k$ then $\val(S_i : i\in k-1)$ is a valuation tree and $\val(S_i : i\in k-1) \subseteq \val(S_i : i\in k)$.
  \item $T_0({<}k) = \val(T_i : i \in k)$ for every $k\in \omega+1$. \qed
\end{enumerate}
\end{observation}

\begin{example}
\begin{figure}
\centering
\includegraphics[scale=0.9]{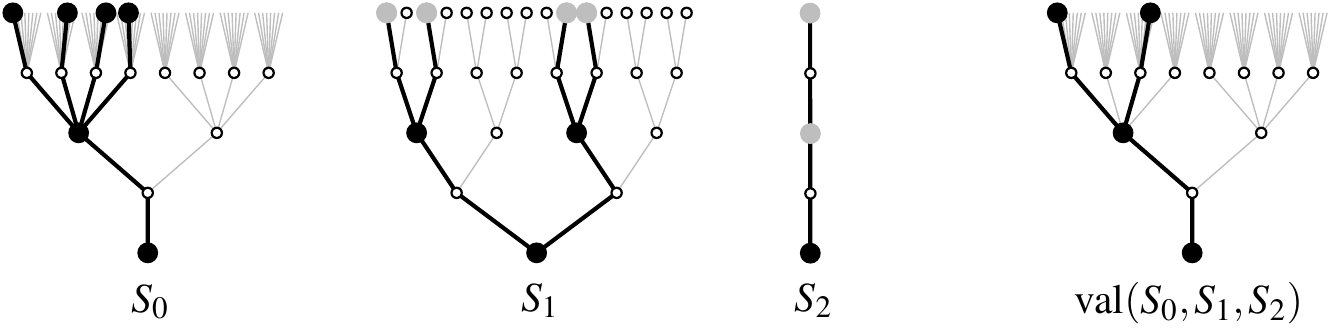}
\caption{A valuation tree (right) constructed from a strong vector subtree (left).}
\label{fig:T2b}
\end{figure}
Figure~\ref{fig:T2b} depicts an example of a valuation tree. In this example, $\sigma = (1,2,1,1,\ldots)$, hence $T_0$ consists of all functions which assign $0$ or $1$ to every decreasing sequence of length 2, $T_1$ is a binary tree and $T_k$ is just a chain for $k\geq 2$. Given a strong vector subtree $(S_0,S_1,S_2)$ of $(T_0,T_1,T_2)$ of height 3 (depicted by thick nodes and thick successor relations) we construct the corresponding valuation tree $\val(S_0,S_1,S_2) \subseteq T_0$. Note that the topmost level of $S_1$ and the two topmost levels of $S_2$ are actually not used in the construction of $\val(S_0,S_1,S_2) \subseteq T_0$. To depict this, they have grey colour.
\end{example}

Let $\sigma$ be an arbitrary signature, consider $\str T = \str T^\sigma$, and let $T$ and $T'$ be subtrees of $T_0$. A function $\psi \colon T \to T'$ is a \emph{structural embedding}
if it is an embedding of trees (preserving meets and relative heights of nodes), and for every $x_0,\ldots,x_{d-1} \in T$ with $\card{x_0} > \cdots> \card{x_{d-1}}$ it holds that
$$\psi(x_0)(\card{\psi(x_1)},\ldots,\card{\psi(x_{d-1})}) = x_0(\card{x_1},\ldots,\card{x_{d-1}}).$$

\begin{observation}\label{obs:structural_emb}
If $\psi\colon T\to T'$ is a structural embedding then it is also an embedding of the structures induced by $T_0$ on $T$ and $T'$ given by Definition~\ref{defn:t0_structure}.\qed
\end{observation}

\begin{lemma}\label{obs:isom}
  Let $\sigma$ be an arbitrary signature and consider $\str T = \str T^\sigma$. For every $k \in \omega +1$ and every valuation subtree $T$ of $T_0$ of height $k$,
  there exists a unique structural embedding $f \colon T_0({<}k) \to T$.
\end{lemma}
\begin{proof}
  For $k\in \omega$ we use induction on $k$.
  Cases $k=0$ and $k=1$ are simple.
  Assume now that the induction hypothesis holds for some $k > 0$ and let $T = \val(S_0,\ldots,S_k)$ have height $k+1$. Put $T' = \val(S_1,\ldots,S_k)$ (it has height $k$).
  By the induction hypothesis there exists 
  a unique structural embedding $f \colon T_0({<}k) \to T({<}k)$ and a unique structural embedding $f' \colon T_1({<}k) \to T'$.

  We will now extend $f$ to $T_0(k)$. Fix $u\in T_0(k)$ and denote by $e \colon k+1 \to \omega$ the increasing enumeration of $L_{T_0}(T)$. Let $v$ be the predecessor of $u$ on level $k-1$, let $w\in T_1(k-1)$ be the unique node such that $u\in v\conc w$ and denote $x=u(k-1)$. Then there is a unique $u' \in \ImmSucc_{S_0}(f(v)) \cap \Succ_{T_0}(f(v)\conc f'(w))$ with $u'(e(k-1)) = x$. Put $f(u)=u'$.
  It is easy to check that the extended map is a structural embedding $T_0({<}k+1) \to T$
  and that the extension was defined in the unique possible way.
  
  If $k = \omega$, we have proved that there are structural embeddings
  $f_i \colon T_0({<}i) \to T({<}i)$ for each $i \in \omega$.
  Since these structural embeddings are unique, we get $f_i \subset f_j$ for $i < j$,
  and $f = \bigcup\{f_i : i \in \omega \}$ is the desired structural embedding.
  On the other hand, if $g \colon T_0 \to T$ is a structural embedding
  then for each $i \in \omega$ the restriction $g \restriction  T_0({<}i): T_0({<}i) \to T({<}i)$ is a structural embedding
  and thus has to be equal to $f_i$, 
  consequently $g = f$.
\end{proof}
\begin{corollary}
For every signature $\sigma$, $k\in \omega+1$ and a valuation tree $\val(S_i : i\in k)$, the number of nodes of $\val(S_i : i\in k)$ depends only on $k$ and $\sigma$.\qed
\end{corollary}

\section{Big Ramsey degrees for $L$-hypergraphs}\label{sec:lhypergraphs}

Most of this section will be spent proving the following proposition which will be the key ingredient in the proof of Theorem~\ref{thm:structures}.
\begin{prop}\label{prop:main}
Let $L$ be a relational language which consists of no unary relations and finitely many relations of every arity, and let $\str H$ be a countable $L$-hypergraph universal for all countable $L$-hypergraphs. Then $\str H$ has finite big Ramsey degrees.
\end{prop}

For this section, fix such $L$ and $\str H$. Assume that $H = \omega$. Let $n_i$ be the number of relations of arity $i$ and assume that $L = \{ \rel{}{i,j} : i\in \omega, 1\leq j \leq n_i \}$ where $\rel{}{i,j}$ has arity $i$. Let $\mu$ be the largest integer for which $n_\mu > 0$, or $\omega$ if such $\mu$ does not exist. Let $\sigma$ be the signature defined by 
$$\sigma_i = \begin{cases}
  n_{i+1}+2 &\text{ if } i < \mu\\
  1 &\text{ otherwise}.
\end{cases}$$
and let $\str T = \str T^\sigma$. Let $\str G$ be the $L$-hypergraph $\str G^\sigma$ on $T_0$ given by Definition~\ref{defn:t0_structure}, and recall that by Observation~\ref{obs:structural_emb}, every structural embedding $T_0\to T_0$ induces an embedding $\str G\to \str G$.

Our aim is to embed $\str H$ into $\str G$, transfer colourings of substructures of $\str H$ into colourings of vector subtrees of $\str T$ and use Milliken's theorem to obtain the desired Ramsey result. However, not all embeddings $\str H\to\str G$ are created equal, and in order to construct one with the properties we need, we first need to introduce some terminology.
\begin{definition}
Let $f$ be a valuation function of level $n$. Given $\bar{x} = (x_0,\ldots,x_{m-1})$ with $n >  x_0 > \cdots > x_{m-1} \geq 0$, we denote by $f^{\bar{x}}$ the valuation function of level $x_{m-1}$ defined by $f^{\bar{x}}(\bar{y}) = f(\bar{x}\conc \bar{y})$ and call it the \emph{$\bar{x}$-slice of $f$}. An \emph{$m$-slice of $f$} is an arbitrary $\bar{x}$-slice of $f$ where $\lvert\bar{x}\rvert = m$. If $m=1$ and $\bar{x} = (x_0)$ we may write $f^{x_0}$ for $f^{\bar{x}}$.
\end{definition}
In particular, a $0$-slice of $f$ is just $f$. Note that the elements of $T_i$ are precisely the $i$-slices of elements of $T_0$.

\begin{definition}\label{defn:enveloping}
Let $k\in \omega$, let $\varphi\colon \omega\to T_0$ be an embedding $\str H\to \str G$ and denote by $R = \varphi[\omega]$ its range. We say that $\varphi$ is \emph{$k$-enveloping} if there is a partition $\omega = O\cup B$ with $O\cap B=\emptyset$ such that the following holds for every $f \in R$, for every $m<k$ and for every $\bar{x} = (x_0,\ldots,x_{m-1})$:
\begin{enumerate}
\item\label{defn:enveloping:1} If $f^{\bar{x}}$ is not constant zero then $x_i\in O$ for every $i$, and
\item\label{defn:enveloping:2} if $g\in R$ or $g$ is constant zero, $g'$ is an $m$-slice of $g$, and $f^{\bar{x}}$ and $g'$ are incomparable then $\card{f^{\bar{x}}\wedge g'}\in B$.
\end{enumerate}
\end{definition}
The letter $O$ stands for the \emph{original} levels, while $B$ is the set of \emph{branching} levels.

\begin{observation}A $k$-enveloping embedding is also $k'$-enveloping for every $k'<k$.\qed\end{observation}

The intuition behind Definition~\ref{defn:enveloping} is that, when embedding $\str H\to \str G$, we ``allocate'' several levels before each image of a vertex of $\str H$ for branching, so that we can have strong control over what happens on meets of slices of vertices from the image of $\str H$. Often (e.g. when there is an upper bound on the number of relations of every arity) it is possible to get an embedding which is $k$-enveloping for every $k$. However, when the number of relations increases with arity, we were so far unable to get such a uniform embedding (and we conjecture that it does not exist, see Conjecture~\ref{conj:no_enveloping}).

\begin{lemma}\label{lem:enveloping}
For every $k\in \omega$, there is a $k$-enveloping embedding $\varphi\colon \str H\to \str G$.
\end{lemma}
\begin{proof}
Recall that the vertex set of $H$ is $\omega$. For every $m\geq 1$, put $$S_m = \{1,\ldots,\max_{i \in k}\{\sigma_{i+m}\}-2\},$$ define
$$J_0 =  \{\branch{s}{x_0,\ldots,x_{m-1}} : s\in S_m , x_0 > \cdots > x_{m-1} \geq 0, m\geq 1\},$$
where we treat $\branch{s}{x_0,\ldots,x_{m-1}}$ as a formal expression, and put
$$J = \omega \cup J_0.$$
Define a linear order $\triangleleft$ on $J$, putting $i \triangleleft j$ if and only if $i < j$, putting $i\triangleleft \branch{s}{x_0,\ldots,x_{m-1}}$ if and only if $i < x_0$ and putting $\branch{s}{\bar{x}} \triangleleft \branch{t}{\bar{y}}$ if and only if either $\bar{x}$ is lexicographically smaller than $\bar{y}$, or $\bar{x} = \bar{y}$ and $s < t$.

Let $\psi\colon (J,\triangleleft) \to (\omega,<)$ be the increasing function such that $\psi(w) = \lvert \{w'\in J : w' \triangleleft w\}\rvert$ (that is, the $\psi(w)$-th element of $J$ is $w$).
Put $O = \psi[\omega]$ and $B = \psi[J_0]$ and note that $O\cap B = \emptyset$ and $O\cup B = \omega$.

Let $\varphi \colon \str{H}\to \str{G}$ be the map where $\varphi(i)\in T_0$ is the valuation function with $\lvert\varphi(i)\rvert = \psi(i)$ and all entries equal to $0$ except for the following two cases:
\begin{enumerate}
  \item\label{l:env:1} If $i > x_0 > \cdots > x_{m-1}$ and $\{i,x_0,\ldots,x_{m-1}\} \in \rel{H}{m+1,r}$ then we have $\varphi(i)(\psi(x_0),\ldots,\psi(x_{m-1})) = r$.
  \item If $m \in k$, $i > x_0 > \cdots > x_{m-1} > y_0 > \cdots > y_{n-1},$
    and $$\{i,x_0,\ldots,x_{m-1},y_0,\ldots,y_{n-1}\} \in \rel{H}{n+m+1,r}$$
    then we have $$\varphi(i)(\psi(x_0),\ldots,\psi(x_{m-1}),\psi(\branch{r}{y_0,\ldots,y_{n-1}}))=\sigma_{m+1}-1.$$
\end{enumerate}
Part~(\ref{l:env:1}) ensures that $\varphi$ is an embedding. We will prove that it is in fact $k$-enveloping, that is, we will verify points~(\ref{defn:enveloping:1}) and~(\ref{defn:enveloping:2}) of Definition~\ref{defn:enveloping}.

Point~(\ref{defn:enveloping:1}) follows straightforwardly from our construction of $\varphi$: The only tuples not valuated by $0$ either consist of levels from $O$ only, or the smallest level is from $B$.

To see point~(\ref{defn:enveloping:2}), suppose that $f' = f^{\bar{x}}$, $g'=g^{\bar{y}}$, $\lvert\bar{x}\rvert = \lvert\bar{y}\rvert < k$, and $f'$ and $g'$ are incomparable. This means that there is a level $n$ such that $f|_n = g'|_n$, but $f'|_{n+1} \neq g'|_{n+1}$. That is equivalent to $n$ being the least integer for which there exists a (possibly empty) decreasing sequence $\bar{z}$ such that $$f(\bar{x}\conc n \conc \bar{z})\neq g(\bar{y}\conc n \conc \bar{z}).$$

For a contradiction suppose that $n\in O$, that is, $n = \psi(n')$ for some $n'\in \omega$. Let $\bar{z} = (z_0,\ldots,z_{p-1})$ and denote $a = f(\bar{x}\conc n \conc \bar{z})$ and $b = g(\bar{y}\conc n \conc \bar{z})$. We know that $a\neq b$, and without loss of generality we can assume that $a\neq 0$. We also know that $z_0,\ldots,z_{p-2}\in O$ (by the construction of $\varphi$, only the last level of a tuple valuated by a non-zero integer can be from $B$), so $\psi^{-1}(z_i)$ is defined for every $0\leq i\leq p-2$.

First suppose that $p\geq 1$ and $z_{p-1}\in O$ or that $p=0$. If $p=0$, let $\bar{z}'$ be the empty tuple, and if $p\geq 1$, put $\bar{z}' = (\psi^{-1}(z_0),\ldots,\psi^{-1}(z_{p-1}))$. It follows from the construction that $\branch{a}{n'{} \conc \bar{z}'} \in J_0$ and
$$f(\bar{x}\conc \psi(\branch{a}{{n'} \conc \bar{z}'})) = \sigma_{\lvert\bar{x}\rvert + 1}-1\neq 0 = g(\bar{x}\conc \psi(\branch{a}{{n'} \conc \bar{z}'})),$$
a contradiction with minimality of $n$ as $\psi(\branch{a}{n'{} \conc \bar{z}'}) < \psi(n') = n$.

So $p\geq 1$ and $z_{p-1} = \psi(\branch{c}{w_0,\ldots,w_{q-1}})$. In this case put 
$$\bar{z}' = (\psi^{-1}(z_0),\ldots,\psi^{-1}(z_{p-2}),w_0,\ldots,w_{q-1}).$$
By the construction we know that
$$f(\bar{x}\conc \psi(\branch{c}{n'{} \conc \bar{z}'})) = \sigma_{\lvert\bar{x}\rvert+1} - 1 \neq 0 = g(\bar{x}\conc \psi(\branch{c}{n'{} \conc \bar{z}'})),$$ hence we again get a contradiction with minimality of $n$, which verifies that point~(\ref{defn:enveloping:2}) is satisfied and thus $\varphi$ is indeed $k$-enveloping.
\end{proof}

\begin{remark}
Note that if there is an absolute bound $N$ on the number of relations of any arity then there exists an embedding which is $k$-enveloping for every $k\in \omega$. Indeed, one can pretend in the above proof that $k = \omega$ and use the fact that we always have $\max_{i \in \omega}\{\sigma_{i+m}\} \leq N$ for any $m\in \omega$. When there is no such bound, we have been unable to produce such an embedding and we conjecture that actually there is no such embedding (in fact, we conjecture a stronger statement, see Conjecture~\ref{conj:no_enveloping}).
\end{remark}

\subsection{Envelopes}
As was noted earlier, our goal is to transfer colourings of substructures of a nice copy of $\str H$ in $\str G$ to colourings of valuation subtrees of $T_0$. 
\begin{definition}
Given $S\subseteq T_0$ and a valuation subtree $T\subseteq T_0$, we say that $T$ is an \emph{envelope} of $S$ if $S\subseteq T$. 
\end{definition}

Having an enveloping embedding allows us to envelope finite subsets of its range in bounded-height valuation trees:

\begin{lemma}\label{lem:envelopes}
For every $k \in \omega$ there exists $R(k) \in \omega$ such that for every $k$-enveloping embedding $\varphi\colon \str H\to \str G$ and every set $S\subset H$ of size $k$ it holds that $\varphi[S]$ has an envelope of height at most $R(k)$.
\end{lemma}
\begin{proof}
Let $O,B\subseteq \omega$ be sets witnessing that $\varphi$ is $k$-enveloping. Given an $i$-slice $f^{\bar{x}}$, we will call it \emph{original} if $\bar{x}\in O^i$.

Fix a set $S$ of $k$ elements from $H$. Put $E_0^0 = \varphi[S]$, $E_0^1 = E_0^0 \cup \{z\}$, where $z$ is the constant zero valuation function of level $\max\{\card{f} : f\in E_0^0\}$, and $E_0^2 = \{ f \wedge_{T_0} g : f,g\in E_0^1 \}\subset T_0$. Define by induction sets $E_i^0$, $E_i^1$ and $E_i^2$ for $1\leq i < R(k)$, where $R(k)$ will be defined later, as follows (recall that $f^{\card{g}}$ is the $\bar{x}$-slice of $f$ for $\bar{x}=(\card{g})$):
\begin{enumerate}
  \item $E_i^0 = \{ f^{\card{g}} : f,g\in E_{i-1}^2, \card{g} < \card{f}\}$,
  \item $E_i^1 = E_i^0 \cup \{z\}$, where $z$ is the constant zero valuation function of level $\max\{\card{f} : f\in E_i^0\}$,
  \item $E_i^2 = \{ f \wedge_{T_i} g : f,g\in E_i^1 \}$.
\end{enumerate}
\begin{claim}\label{c:envelopes}
The following properties hold for every $0\leq i < R(k)$:
\begin{enumerate}
  \item\label{env:1} $E_i^2\supseteq E_i^1\supseteq E_i^0$ and $E_i^2$ is a subtree of $T_i$,
  \item\label{env:2} each element of $E_i^1$ is either constant zero or an original $i$-slice of a member of $\varphi[S]$ and each element of $E_i^2$ is a restriction of an element of $E_i^1$,
  \item\label{env:3} there are at most $\max(0,k - i)$ levels with non-zero members of $E_i^1$,
  \item\label{env:4} for $i>0$, it holds that $L_{T_i}(E_i^2) \setminus L_{T_{i-1}}(E_{i-1}^2) \subseteq B$,
  \item\label{env:5} if $\ell \in L_{T_i}(E_i^2) \cap O$ then there is $f\in \varphi[S]$ with $\lvert f\rvert = \ell$.
  \item\label{env:6} for $i>0$ we have that $L_{T_{i-1}}(E_{i-1}^2) \setminus L_{T_i}(E_i^2) = \{\max(L_{T_{i-1}}(E_{i-1}^2))\}$,
  \item\label{env:7} for $i>0$, it holds that $\max(L_{T_i}(E_i^2)) < \max(L_{T_{i-1}}(E_{i-1}^2))$.
\end{enumerate}
\end{claim}
We will proceed by induction on $i$. For $i=0$ this is immediate. So $i > 0$ and we know that all properties hold for $i-1$. Property~(\ref{env:1}) is straightforward from the definition. We know that (\ref{env:2}) holds for $i-1$, so $E_i^0$ consists of slices of members of $\varphi[S]$ and constant zero functions, so this is true also for $E_i^1$, and in constructing $E_i^2$ we are only adding restrictions of members of $E_i^1$.

Note that if $i>k$ then, by~(\ref{env:3}), we know that we only had constant zero functions in $E_{i-1}^1$ and consequently also in $E_{i-1}^2$ and $E_i^1$. So it suffices to prove~(\ref{env:3}) for $i\leq k$. By~(\ref{env:5}) for $i-1$ we have that the only levels from $O$ we can slice with correspond to members of $\varphi[S]$ and by (\ref{env:2}) we know that the only non-zero members of $E_i^1$ are original $i$-slices of members of $\varphi[S]$. To get an $i$-slice, we need $i+1$ members of $\varphi[S]$ (one to slice and the other $i$ for a decreasing sequence of levels from $O$ to slice with), hence there are only $k-i$ possible last elements of the decreasing sequence. This verifies~(\ref{env:3}).

To see~(\ref{env:4}), note that we only add new levels compared to $E_{i-1}^2$ when taking meets in the construction of $E_i^2$. This means that we can assume that $i < k$ as otherwise all elements of $E_i^1$ are comparable. By~(\ref{env:2}) we know that each member of $E_i^1$ is either constant zero or an original $i$-slice of a member of $\varphi[S]$, and since $\varphi$ is $k$-enveloping and $i < k$ it follows that their meets happen on levels from $B$.

Property~(\ref{env:5}) is immediate from~(\ref{env:4}) and the induction hypothesis. Property~(\ref{env:6}) is also easy, in the construction of $E_i^0$ we only lose the highest level. Having~(\ref{env:6}),~(\ref{env:7}) is again easy, because we only add new levels as meets of functions from existing levels. This finishes the proof of Claim~\ref{c:envelopes}.

\medskip

Property~(\ref{env:3}) implies that each member of $E_k^1$ is constant zero, hence $E_k^2 = E_k^1$. It follows that for every $i > k$ we have $E_i^1 = E_i^2$ and all its elements are constant zeros.

Put $\mathfrak L = \bigcup_{i=0}^k L_{T_i}(E_i^2)$. Clearly, for every $i\geq 1$ it holds that $\lvert L_{T_i}(E_i^1)\rvert \leq \lvert L_{T_{i-1}}(E_{i-1}^2)\rvert$ and for every $i\geq 0$ we have $\lvert L_{T_i}(E_i^2)\rvert \leq 2\lvert L_{T_i}(E_i^1)\rvert -1$. Together with $\lvert L_{T_0}(E_0^1)\rvert \leq k$ this implies that $\lvert\mathfrak L\rvert$ is bounded from above by some $R(k)$ which is a function of $k$. In particular, $\lvert L_{T_j}(E_j^2)\rvert \leq R(k) - j$, and so $E_{R(k)-1}^2$ is a singleton set containing a constant zero function.

For every $i\in R(k)$ define $$E_i^3 = \{ f|_\ell : f\in E_i^2, \ell \in \mathfrak L, \card{f}\geq \ell \}.$$ Note that $L_{T_{i}}(E_i^3) \subseteq \mathfrak L$ and that $E_i^3$ is a subtree of $T_i$ for every $i\in R(k)$. For every $i\in R(k)$ let $S_i$ be some strong subtree of $T_i$ containing $E_i^3$ such that $L_{T_i}(S_i) = \mathfrak L$. Such trees exist by Observation~\ref{obs-subtree}.

It remains to prove that $\varphi[S] \subseteq \val(S_i : i\in R(k))$. We will prove the following stronger result:
\begin{claim}\label{c:vals}
For every $i\in R(k)$ and $\ell\in \mathfrak L$ it holds that if $f\in E_{R(k)-i-1}^3$ and $\ell\leq\card{f}$ then $f|_\ell \in\val(S_{R(k)-i-1},\allowbreak\ldots,  S_{R(k)-1})$.
\end{claim}
We will prove this statement by double induction on $i$ (outer induction) and $\ell\in \mathfrak L$ (inner induction).
For $i=0$ this is easy because $E_{R(k)-1}^3$ consists only of the root of $S_{R(k)-1}$ which is also the constant zero function of level $\min(\mathfrak L)$ and is the root of $\val(S_{R(k)-1})$.

We will now prove the statement for $i$ and $\ell$, assuming that it holds for every $j < i$ and every $\ell'\in\mathfrak L$, and also for $i$ and every $\ell'\in \mathfrak L$ such that $\ell' < \ell$. Pick an arbitrary $f\in E_{R(k)-i-1}^3$. If $\ell = \min(\mathfrak L)$ then $f|_\ell$ is the root of $S_{R(k)-i-1}$ and hence the root of $\val(S_{R(k)-i-1}, \ldots,S_{R(k)-1})$.

So $\ell > \min(\mathfrak L)$. Let $\ell'$ be the largest member of $\mathfrak L$ smaller than $\ell$. By the construction, $f|_{\ell'} \in E_{R(k)-i-1}^3$, and so $f|_{\ell'} \in \val(S_{R(k)-i-1}, \ldots,S_{R(k)-1})$ by the induction hypothesis for $i$ and $\ell'$.

If $\ell'\in B$ then $f^{\ell'}$ is the constant zero valuation function of level $\ell'$ and as such is in $E_{R(k)-i}^3$ and consequently in $\val(S_{R(k)-i}, \ldots,S_{R(k)-1})$ by the induction hypothesis. Otherwise $\ell' \in O$, but then $f^{\ell'} \in E_{R(k)-i}^0 \subseteq E_{R(k)-i}^3$ and hence also in $\val(S_{R(k)-i}, \ldots,S_{R(k)-1})$ by the induction hypothesis.

So we know that 
$$f|_{\ell'} \in \val(S_{R(k)-i-1}, \ldots,S_{R(k)-1})$$
and
$$f^{\ell'} \in \val(S_{R(k)-i}, \ldots,S_{R(k)-1}).$$ The definition of valuation tree then gives that $f|_\ell \in\val(S_{R(k)-i-1}, \ldots,S_{R(k)-1})$ which concludes the proof of Claim~\ref{c:vals}.

\medskip

As a special case of Claim~\ref{c:vals} we get that $\varphi[S] = E_0^0\subseteq E_0^3 \subseteq \val(S_i : i\in R(k))$. This means that there indeed is a valuation tree of height at most $R(k)$ which contains $\varphi[S]$.
\end{proof}

Above we conjectured that if the sequence $\sigma$ is unbounded, one cannot construct an embedding $\str H\to \str G$ which would be $k$-enveloping for every $k$. In fact, we conjecture that in such a setting no single embedding $\str H \to \str G$ can have bounded envelopes of all finite subsets:

\begin{conjecture}\label{conj:no_enveloping}
If for every $n\in\omega$ there is $a\in\omega$ such that the number of relations of arity $a$ is at least $n$ then for every embedding $\varphi\colon\str H\to \str G$ there is some $k\in \omega$ such that for every $R\in \omega$, there is a set $S\subseteq \varphi[\str H]$ with $\lvert S\rvert = k$ whose all envelopes have height at least $R$.
\end{conjecture}

We can now proceed with a proof of Proposition~\ref{prop:main}.

\begin{proof}[Proof of Proposition~\ref{prop:main}]
Fix a finite $L$-hypergraph $\str A$ with $\lvert A\rvert = k$, a $k$-enveloping embedding $\varphi\colon \str H\to \str G$ (for example one given by Lemma~\ref{lem:enveloping}), and a colouring $\chi_0 \colon {\str{H}\choose\str{A}}\to p$. Since $\str{H}$ is universal, it follows that there is an embedding $\theta \colon \str{G}\to \str{H}$. Consider the colouring $\chi\colon{\str{G}\choose\str{A}} \to p$ obtained by setting
$\chi\left(f\right) = \chi_0\left(\theta \circ f\right)$ for every $f \in \binom{\str{G}}{\str{A}}$.

Let $h = R(k)$ be given by Lemma~\ref{lem:envelopes} and let $\str{G}_h$ be the induced sub-$L$-hypergraph of $\str{G}$ on $T_0({<}h)$. We enumerate ${\str{G}_h\choose\str{A}} = \{ f_i : i \in \ell \}$ for some $\ell \in \omega$.

By Lemma~\ref{obs:isom}, for every valuation tree $T$ of height $h$,
there is a structural embedding $f_T \colon T_0({<}h) \to T$ which is also an isomorphism from $\str G_h$ to the substructure of $\str G$ induced on $T$.
Let $\str{S}=(S_i : i\in h)$ be a strong vector subtree of $\str{T}\restriction_h$ of height $h$ and
consider the structural embedding $f = f_{\val(S_i : i\in h)} \colon \str{G}_h \to \val(S_i : i\in h)$.
Put \[\bar\chi(\str{S}) = \left\langle\, \chi\left(f \circ f_i\right) : i \in \ell \,\right\rangle,\]
which is a finite colouring of $\Str_h(\str{T}\restriction_h)$.
By Theorem~\ref{thm:Milliken}, there is an infinite strong vector subtree of $\str{T}\restriction_h$ which is monochromatic with respect to $\bar\chi$.
Extend it arbitrarily to an infinite strong vector subtree of $\str T$ with the same level set and let $U$ be its corresponding valuation subtree. Note that the extension does not influence in any way the valuation subtrees of height $h$.
The structural embedding $\psi \colon T_0\to U$ given by Lemma~\ref{obs:isom} is a hypergraph embedding $\psi \colon \str{G}\to \str{G}$.

We claim that $\theta\circ \psi \circ \varphi[\str{H}]$ is the desired copy $g[\str{H}]$ of $\str{H}$, in which copies of $\str{A}$ have at most $\ell$ different colours in~$\chi_0$. This is true, because by Lemma~\ref{lem:envelopes} every copy of $\str A$ in $\varphi[\str H]$ is contained in a valuation subtree of height $h$. All of these subtrees have the same colour (with respect to $\bar\chi$) in $\psi\circ \varphi[\str H]$, and so we know that $\chi$ takes at most $\ell$ different values on $\psi\circ \varphi[\str H]\choose\str{A}$. Consequently, $\chi_0$ attains at most $\ell$ different values on $\theta\circ \psi \circ \varphi[\str{H}] \choose \str A$.
\end{proof}

\section{The main results}
In this section we do three simple constructions on top of Proposition~\ref{prop:main} in order to prove Theorem~\ref{thm:structures}.

\subsection{Unary relations}
First, we introduce a general construction for adding unary relations in order to prove the following theorem.
\begin{theorem}\label{thm:main}
Let $L$ be a relational language with finitely many relations of every arity greater than one and with finitely or countably many unary relations. Let $\str H$ be a countable $L$-hypergraph universal for all countable $L$-hypergraphs. Then $\str H$ has finite big Ramsey degrees.
\end{theorem}
\begin{proof}
Without loss of generality we can assume that every vertex of $\str H$ is in exactly one unary relation. Assume that the unary relations of $L$ are $\{U^i : i\in u\}$ for some $u\in \omega+1$.

Let $L^-$ be the language which one gets from $L$ by removing all unary relations and let $\str M$ be a countable $L^-$-hypergraph universal for all countable $L^-$-hypergraphs (it exists for example by the \Fraisse{} theorem~\cite{Fraisse1953}). Without loss of generality we assume that the vertex set of both $\str M$ and $\str H$ is $\omega$. We now define an $L$-hypergraph $\str G$ as follows:
\begin{enumerate}
\item The vertex set of $\str G$ is $G = \{(v, i) : v\in \omega, i\in \min(v+1, u)\}$ with enumeration given by the lexicographic order,
\item vertex $(v,i)$ is in unary relation $U^j$ if and only if $i=j$, and
\item $((x_0,i_0),\ldots,(x_n,i_n))\in\rel{G}{}$ if and only if $(x_0,\ldots,x_n) \in \nbrel{\str M}{}$.
\end{enumerate}

Fix some finite $L$-hypergraph $\str A$ and let $\str A^-$ be its $L^-$-reduct. By Proposition~\ref{prop:main} there is $\ell\in \omega$ such that $\str M \longrightarrow (\str M)^{\str A^-}_{k,\ell}$ for every $k\in \omega$. We will now prove that $\str G\longrightarrow (\str H)^{\str A}_{k,\ell}$ for every $k\in \omega$.

Let $\pi\colon G\to M$ be the map sending $(v,i)\mapsto v$. We say that a map $f\colon X\to G$ is \emph{transversal} if $\pi\circ f$ is injective. Note that if $f\colon \str A\to \str G$ is a transversal embedding then $\pi\circ f$ is an embedding $\str A^-\to \str M$.
A colouring $\chi_0\colon {\str G\choose \str A} \to k$ thus induces a partial colouring $\chi\colon {\str M\choose {\str A^-}}\to k$ by ignoring non-transversal copies and composing with $\pi$. There may be copies of $\str A^-$ which do not get any colour, we assign them a colour arbitrarily.

Since $\str M \longrightarrow (\str M)^{\str A^-}_{k,\ell}$, there is an embedding $\psi\colon \str M\to \str M$ with $\chi$ attaining at most $\ell$ colours on ${{\psi[\str M]}\choose {\str A^-}}$. Let $\psi_0\colon \str G\to \str G$ be the embedding mapping $(v,i)\mapsto (\psi(v),i)$. We have that $\chi_0$ attains at most $\ell$ colours on transversal copies from ${{\psi_0[\str G]}\choose \str A}$.

Let $\str H^-$ be the $L^-$-reduct of $\str H$ forgetting unary relations. Since $\str H^-$ is a countable $L^-$-hypergraph, there is an embedding $\varphi_0\colon \str H^-\to \str M$ with the property that $\varphi_0(x) \geq i_x$ for $x\in U^{i_x}_\str H$. It is straightforward to check that the map $\varphi\colon H\to G$, defined by $\varphi(x) = (\varphi_0(x), i_x)$, is a transversal embedding $\str H\to \str G$, hence $\chi_0$ attains at most $\ell$ colours on ${{\psi_0\circ\varphi[\str H]}\choose \str A}$. Consequently, $\str G\longrightarrow (\str H)^{\str A}_{k,\ell}$.

Knowing that $\str G\longrightarrow (\str H)^{\str A}_{k,\ell}$, proving that $\str H\longrightarrow (\str H)^{\str A}_{k,\ell}$ is straightforward: $\str G$ is a countable $L$-hypergraph and so there is an embedding $\theta\colon \str G\to \str H$, which means that a colouring of ${\str H\choose \str A}$ restricts to a colouring of ${\str G\choose \str A}$ exactly as in the proof of Proposition~\ref{prop:main}.
\end{proof}

\subsection{Non-$L$-hypergraphs}
For the constructions above it was convenient to work with $L$-hypergraphs. It is folklore that one does not lose any generality working with $L$-hypergraphs only (when structures are enumerated), we give a proof for completeness.
\begin{prop}\label{prop:non_hypergraphs}
Let $L$ be a relational language with finitely many relations of every arity greater than one and with finitely or countably many unary relations, let $\mathcal C$ be the class of all countable $L$-structures where every relation is injective and every vertex is in exactly one unary relation, and let $\str H\in \mathcal C$ be universal for $\mathcal C$. Then $\str H$ has finite big Ramsey degrees.
\end{prop}
Note that this statement is the same as the statement of Theorem~\ref{thm:structures} for $\mathcal F = \emptyset$ with the extra technical assumption that every vertex is in exactly one unary relation.
\begin{proof}
For every $i\in \omega$, put
$$M_i = \{(\rel{}{},\pi) : \rel{}{}\in L, i = \arity(\rel{}{}), \pi\in \Sym(i)\}.$$
Given $\str A\in \mathcal C$ and vertices $x_1 < \cdots < x_n \in \str A$, put 
$$M_\str A(x_0,\ldots,x_{n-1}) = \{(\rel{}{},\pi)\in M_n : (x_{\pi(0)},\ldots,x_{\pi(n-1))}) \in \rel{A}{}\}.$$
Define a language $L'$ containing the same unary relations as $L$ and for every $i\geq 2$ and every nonempty $S\subseteq M_i$ an $i$-ary relation $\rel{}{S}$. Note that $L'$ has only finitely many relations of every arity greater than one and finitely or countably many unary relations. Let $\mathcal C'$ be the class of all $L'$-hypergraphs. 

Define $T$ to be the map assigning to every $\str A\in \mathcal C$ a structure $T(\str A)\in \mathcal C'$ on the same vertex set with the same unary relations such that for every $x_0<\cdots<x_{n-1}\in A$ with $n\geq 2$ we have
$$\{x_0, \ldots, x_{n-1}\} \in \nbrel{T(\str A)}{M_\str A(x_0,\ldots,x_{n-1})}$$ if $M_\str A(x_0,\ldots,x_{n-1})$ is nonempty. There are no other relations in $T(\str A)$. Define $U$ to be the map assigning to every $\str A\in \mathcal C'$ a structure $U(\str A)\in \mathcal C$ on the same vertex set with the same unary relations such that whenever we have $x_0 < \cdots < x_{n-1} \in \str A$ with $\{x_0,\ldots,x_{n-1}\} \in \rel{A}{S}$, we put
$$(x_{\pi(0)},\ldots,x_{\pi(n-1))}) \in \nbrel{U(\str A)}{}$$
for every $(\rel{}{},\pi) \in S$. There are no other relations in $U(\str A)$.

It is straightforward to verify $T$ and $U$ are mutually inverse and define a bijection between $\mathcal C$ and $\mathcal C'$. Moreover, given $\str A,\str B\in \mathcal C$ and a function $f\colon A\to B$, it holds that $f$ is an embedding $\str A\to \str B$ if and only if it is an embedding $T(\str A)\to T(\str B)$. Consequently, the big Ramsey degree of $\str A$ in $\str H$ is the same as the big Ramsey degree of $T(\str A)$ in $T(\str H$), hence finite by Theorem~\ref{thm:main}.
\end{proof}

\subsection{Forbidding structures}
\begin{proof}[Proof of Theorem~\ref{thm:structures}]
Let $\mathcal F$ be the family witnessing that $\str H$ is unrestricted. Since $\str H$ is countable, there are only countably many types of vertices. Hence, without loss of generality, we can assume that each vertex of $\str H$ is in exactly one unary relation and that $\mathcal F$ forbids no single unary relation (otherwise we can remove it from the language).

Let $\mathcal C$ be the class of all countable $L$-structures where all relations are injective and every vertex is in exactly one unary relation, and let $\str M \in \mathcal C$ be universal for $\mathcal C$ (such $\str M$ exists for example by the \Fraisse{} theorem~\cite{Fraisse1953}). We say that a subset $S\subseteq M$ is \emph{bad} if $\str M$ induces a structure from $\mathcal F$ on $S$ and we say that a tuple $\bar{x}$ of elements of $\str M$ is \emph{bad} if it contains a bad subset. Let $\str G$ be the $L$-structure on the same vertex set as $\str M$ such that $\bar x\in \rel{G}{}$ if and only if $\bar x\in \rel{M}{}$ and $\bar{x}$ is not bad (i.e. we remove bad tuples from all relations). Clearly, $\str G$ is $\mathcal F$-free, and note that $\str G$ and $\str M$ have the same unary relations. We will prove that $\str G$ has finite big Ramsey degrees. Since $\str G$ embeds into $\str H$ and any embedding $\str H\to \str M$ is also an embedding $\str H\to \str G$, this would imply that $\str H$ has finite big Ramsey degrees, thereby proving the theorem.

Fix a finite $\mathcal F$-free $L$-structure $\str A$ where all relations are injective. Let $\iota$ be the identity map understood as a function $G\to M$ and let $\str A_0,\ldots,\str A_m$ be some enumeration of all isomorphism types of structures from $\{\iota\circ f[\str A] : f \text{ is an embedding } \str A\to \str G\}$, that is, it is an enumeration of all possible isomorphism types of structures which, after removing bad tuples from relations, are isomorphic to $\str A$. There are only finitely many of them because they have the same unary relations as $\str A$ and there are only finitely many $L$-structures on a given number of vertices with given unary relations, up to isomorphism.

For every $0\leq i\leq m$, let $\ell_i$ be the big Ramsey degree of $\str A_i$ in $\str M$ ($\ell_i$ is finite by Proposition~\ref{prop:non_hypergraphs}) and put $\ell = \sum_{i=0}^m \ell_i$. We now prove that $\str G\longrightarrow (\str G)^\str A_{k,\ell}$ for every $k\in \omega$.

Fix a colouring $\chi\colon {\str G \choose \str A} \to k$ and let $\chi_i\colon {\str M\choose \str A_i} \to k$, $0\leq i\leq m$, be the colourings obtained from $\chi$ by composing with $\iota$. By inductive usage of Proposition~\ref{prop:non_hypergraphs} we get an embedding $f\colon \str M\to \str M$ such that, for every $0\leq i\leq m$, $\chi_i$ attains at most $\ell_i$ colours on ${f[\str M] \choose \str A_i}$. As all structures in $\mathcal F$ are covered by a relation, it follows from the construction that $f$ is also an embedding $\str G\to \str G$, and since every copy in ${f[\str G] \choose \str A}$ corresponds to a copy of $\str A_i$ in $f[\str M]$ for some $i$, it follows that $\chi$ attains at most $\ell$ colours on ${f[\str G] \choose \str A}$, hence the big Ramsey degree of $\str A$ in $\str G$ is indeed finite.
\end{proof}

\section{Infinite big Ramsey degrees?}\label{sec:inf}
As soon as one has infinite branching, the Milliken theorem stops being true even for colouring vertices (see Proposition~\ref{prop:hl}). This means that one cannot generalise our methods directly for languages with infinitely many relations of some arity at least 2. In fact, no known methods generalise because all of them find very specific \emph{tree-like} copies and one can construct infinite colourings which are persistent on these copies. Doing this in full generality requires developing the theory of \emph{weak types} and it will appear elsewhere. Here we only show a special case which is technically much simpler but only works for binary relations.



\begin{definition}
Let $\str H$ be a countable relational structure with vertex set $\omega$. Given $X\subseteq \omega$, we define the \emph{type of $X$}, denoted by $\tp_\str H(X)$, to be the isomorphism type of the substructure of $\str H$ induced on $X$. (Note that this corresponds to the quantifier-free type over the empty set from model theory if we consider the enumeration to be part of the language.)

Let $f\colon \str H\to \str H$ be an embedding. We say that $f$ is \emph{1-tree-like} (or, in this paper, simply \emph{tree-like}) if for every finite $X = \{x_0 < \cdots < x_m\} \subset \omega$, every $0\leq i \leq m$ and every $x\in \omega$ such that $x > x_m$ there exists $y > x_m$ such that $$\tp_\str H(f[X] \cup \{f(y)\}) = \tp_\str H(X\cup\{x\})$$ and moreover $$\tp_\str H(\{0,\ldots,f(x_0)-1,f(y)\}) = \tp_\str H(\{0,\ldots,f(x_0)-1,f(x_i)\}).$$
\end{definition}

The first condition is satisfied for every embedding by the choice $y=x$. The second condition says that for tree-like embeddings we have some control even over types with respect to the ambient structure $\str H$ within which our copy lies. Note that every structural embedding $\str T^\sigma \to \str T^\sigma$ is a tree-like embedding of the corresponding hypergraphs.

\begin{example}\label{ex:rado}
Let $L$ be a language consisting of countably many binary relations and let $\str R$ be the countable homogeneous $L$-hypergraph (that is, the infinite-edge-coloured countable random graph where the colour classes are generic). One can repeat our constructions for $\str R$ and get the everywhere infinitely branching tree $T = \omega^{<\omega}$ as $T_0$. If the Milliken theorem was true for infinitely branching trees, it would produce tree-like copies (which would arise simply as strong subtrees isomorphic to $\omega^{<\omega}$).

On the other hand, assume that the vertex set of $\str R$ is $\omega$ and consider any embedding $f\colon \str R\to \str R$ such that $f(i)$ is connected to vertex $0$ by the $i$-th relation (such an embedding exists by the extension property). This embedding is, in a way, as far as possible from being tree-like, because everything branches at level $0$ and one cannot say anything about the behaviour of this copy with respect to the external vertices.
\end{example}

All known big Ramsey methods produce tree-like copies, and whenever the exact big Ramsey degrees are known, the proof can be adapted to show that every copy contains a subcopy which is ``weakly tree-like'' (see e.g.~\cite{Balko2021exact,Balko2023,Laflamme2006}). Example~\ref{ex:rado} shows that such a property fails for the infinite-edge-coloured random graph. We believe that working in the category of tree-like embeddings (or some other variant of nice embeddings) is a reasonable thing to do as it still captures (a lot of) the combinatorial complexity of the problems, in addition to allowing one to prove negative results.

In this section we will see an instance of this. Given a relational structure $\str A$, its Gaifman graph is the graph on the same vertex set where vertices $x\neq y\in A$ are connected by an edge if and only if there exists a tuple $\bar{z}$ of vertices of $\str A$ containing both $x$ and $y$ which belongs to a relation of $\str A$. A relational structure $\str A$ is \emph{irreducible} if its Gaifman graph is a complete graph.

\begin{theorem}\label{thm:inf}
Let $L$ be a relational language, let $\mathcal F$ be a set of finite irreducible $L$-structures and let $\str H$ be a countable $\mathcal F$-free structure universal for all countable $\mathcal F$-free structures. Assume that there are $\mathcal F$-free structures $\str B$ and $\str U$ consisting of one vertex, and infinitely many pairwise non-isomorphic 1-vertex $\mathcal F$-free extensions $\str C_0,\str C_1,\ldots$ of $\str B$ such that the added vertices come last in the enumeration and each of them is isomorphic to $\str U$. Then there exists an $\mathcal F$-free structure $\str A$ on three vertices and a colouring $c\colon {\str H\choose \str A}\to \omega$ such that if $f\colon \str H\to \str H$ is a tree-like embedding then $c$ attains all values on ${f[\str H] \choose \str A}$.
\end{theorem}
In other words, the big Ramsey degree of $\str A$ is infinite for tree-like copies. Note that if $\mathcal F$ either contains no enumeration or all enumerations of every finite $\str A$, the class of all finite $\mathcal F$-free structures is a free amalgamation class when one ignores the enumerations.

\begin{corollary}\label{cor:inf}
Let $L$ be a relational language containing infinitely many binary relations and let $\str H$ be a countable $L$-hypergraph universal for all countable $L$-hypergraphs. Then there exists an $L$-hypergraph $\str A$ on three vertices and a colouring $c\colon {\str H\choose \str A}\to \omega$ such that if $f\colon \str H\to \str H$ is a tree-like embedding then $c$ attains all values on ${f[\str H] \choose \str A}$.
\end{corollary}
\begin{proof}
Note that being an $L$-hypergraph can be described by a set $\mathcal F$ of forbidden irreducible structures. Let $\str B$ be the hypergraph on 1 vertex which is in no unary relations and let $\str C_0, \str C_1, \ldots$ be all possible one-vertex extensions of $\str B$ by a vertex in no unary relation. There are infinitely many of them as they correspond to binary relations in $L$. Therefore the conditions of Theorem~\ref{thm:inf} are satisfied and the conclusion follows.
\end{proof}
A particular example of a structure satisfying this corollary is the infinite-edge-coloured random graph from Example~\ref{ex:rado}.

Our proof of Theorem~\ref{thm:inf} is derived from a proof that the Halpern--L\"auchli theorem does not hold for the tree $\omega^{<\omega}$, which we now present as a nice warm-up. We believe that this proof is folkloristic, but we were unable to find it in the literature.

\begin{prop}\label{prop:hl}
Let $T=\omega^{<\omega}$ be the tree of all finite sequences of natural numbers and let $\sqsubseteq$ be the usual tree order by end-extension. There is a colouring $c\colon T\to \omega$ such that whenever $T'$ is a strong subtree of $T$ of infinite height then $c[T'] = \omega$.
\end{prop}
\begin{proof}
Given $t\in T$, we put $w(t) = \lvert t \rvert + \sum_{i < \lvert t\rvert} t(i)$ and we define $\ell(t)$ to be the least $\ell$ such that $w(t|_\ell) \geq \lvert t\rvert$. Note that $\ell(t)$ always exists as $w(t) \geq \lvert t \rvert$.

Define a colouring $c\colon T\to \omega$, putting $c(t) = w(t|_{\ell(t)}) - \lvert t\rvert$. Let $T'$ be an arbitrary strong subtree of $T$ of infinite height, let $r$ be the root of $t'$ and let $n\in \omega$ be such that $n > w(r)$ and $T'\cap T(n) \neq \emptyset$. Put $k = n-w(r)-1$. Now we will prove that $c[T'] = \omega$. For that, fix a colour $x\in \omega$ and find $t\in T'$ such that $\lvert t\rvert = n$ and $r\conc(k+x) \sqsubseteq t$ (such $t$ exists as $T'$ is a strong subtree of $T$). Now, $w(r) < n$, so $\ell(t) > \lvert r\rvert$. On the other hand, $w(r\conc(k+x)) = k+x+w(r)+1 = n+x \geq n$, hence $\ell(t) = \lvert r\rvert + 1$ and $c(t) = x$. This means that for every $x\in \omega$ we can find $t\in T'$ such that $c(t) = x$, hence indeed $c[T'] = \omega$.
\end{proof}
Note that in the construction of the colouring $c$ we essentially only needed to be able to address a particular level $\ell(t)$ on which we knew that the passing numbers attain all possible values. We will now show how this idea can be adapted to find infinite colourings persistent in tree-like embeddings.

\begin{proof}[Proof of Theorem~\ref{thm:inf}]
Let $\str B_0, \str B_1, \ldots$ be an enumeration of all copies of $\str B$ in $\str H$ such that the vertex of $\str B_i$ comes in the enumeration of $\str H$ before the vertex of $\str B_j$ whenever $i < j$. There are infinitely many of them because the infinite disjoint union of copies of $\str B$ is $\mathcal F$-free and thus embeds into $\str H$.

Given a vertex $v\in \str H$ isomorphic to $\str U$ (i.e. it has the same unary relations), we let $b(v)$ be the least integer such that there is $w\in \str B_{b(v)}$ with $w\geq v$, and let $s(v)\colon b(v) \to \omega$ be the sequence satisfying $s(v)_i = x$ if and only if the structure induced by $\str H$ on $\str B_i \cup \{v\}$ is isomorphic to $\str C_x$. 

Given a vertex $v\in \str H$ isomorphic to $\str B$, we let $j(v)$ be such that $\{v\} = B_{j(v)}$. Given $s\in \omega^{<\omega}$ and $n\leq \lvert s\rvert$, we put $w(s) = \lvert s\rvert + \sum_{i\in\lvert s\rvert}s(i)$ and define $\ell(s,n)$ to be the least $\ell$ such that $w(s|_\ell) \geq n$. It exists since $n\leq \lvert s\rvert$.

Let $\str A$ be an $\mathcal F$-free structure on three vertices such that the second vertex in the enumeration is isomorphic to $\str B$, the third one is isomorphic to $\str U$ and there are no non-unary relations in $\str A$. Such $\str A$ clearly exists, for example by making its smallest vertex isomorphic to $\str B$.

For a copy $\widetilde{\str A} \subseteq \str H$, we denote by $n(\widetilde{\str A})$ the level of the first vertex and we put $s(\widetilde{\str A}) = s(v)|_{j(w)}$ where $w$ is the second and $v$ is the third vertex of $\str A$. We define a colouring $c\colon {\str H\choose \str A} \to \omega$ putting
$$c(\widetilde{\str A}) = \begin{cases}
w(s|_{\ell(s,n)}) - n &\text{ if $n\leq \lvert s\rvert$, where }s=s(\widetilde{\str A})\text{ and }n=n(\widetilde{\str A}),\\
0 &\text{ otherwise}.
\end{cases}$$

Let $f\colon\str H\to \str H$ be an arbitrary tree-like embedding of $\str H$ to $\str H$. We will prove that $c[{{f[\str H]} \choose \str A}] = \omega$. Given $p\in \omega$, we will construct a copy $\widetilde{\str A}\in{{f[\str H]} \choose \str A}$ with $c(\widetilde{\str A}) = p$. Assume that the vertex set of $\str H$ is $\omega$ and let $r_0 < r_1$ be arbitrary vertices of $\str H$ such that $r_0$ is isomorphic to $\str B$ and $r_1$ is isomorphic to $\str U$. Let $r_2\in \omega$ be chosen such that:
\begin{enumerate}
  \item $r_2 > r_1$,
  \item $r_2$ has the same unary relation as the first vertex of $\str A$, and
  \item $f(r_2)$ is the $n$-th vertex of $\str H$ for some $n$ with $n > w(s(f(r_1))|_{j(f(r_0))})$.
\end{enumerate}
Such a vertex exists, because the infinite disjoint union of vertices with the unary is $\mathcal F$-free, and so $\str H$ contains infinitely many such vertices. Let $r_3\in \omega$ be an arbitrary vertex such that $r_3 > r_2$, $r_3$ is isomorphic to $\str B$, $j(f(r_3))> n$ and there are no relations on $\{r_2,r_3\}$.

Put $q = n-w(s(f(r_1))|_{j(f(r_0))})-1+p$, $X = \{r_0, r_1, r_2,r_3\}$ and $i=1$. Pick any $x\in \omega$ such that $x > r_3$, $\{r_0, x\}$ is isomorphic to $\str C_q$ and there are no relations on $\{r_2,x\}$ and $\{r_3,x\}$. Use tree-likeness of $f$ for $X$, $i$ and $x$ to obtain $y\in \omega$ such that $y > r_3$, 
$$\tp_\str H(\{f(r_0), f(r_2), f(r_3), f(y)\}\}) = \tp_\str H(\{r_0, r_2, r_3, x\})$$ and 
$$\tp_\str H(\{0,\ldots,f(r_0)-1,f(y)\}) = \tp_\str H(\{0,\ldots,f(r_0)-1,f(r_1)\}).$$
Put $v = f(y)$ and observe that $v$ satisfies the following:
\begin{enumerate}
  \item $\{f(r_0),v\}$ is isomorphic to $\str C_p$, so in particular $v$ is isomorphic to $\str U$,
  \item there are no relations on $\{f(r_2),v\}$ and $\{f(r_3),v\}$,
  \item $v > f(r_3)$ and thus $b(v) > j(f(r_3)) > n$,
  \item $s(f(r_1))|_{j(f(r_0))}\sqsubseteq s(v)$,
  \item $s(v)_{j(f(r_0))} = q$, and
  \item $\str H$ induces a copy of $\str A$ on $\{f(r_2),f(r_3),v\}$.
\end{enumerate}

Let $\widetilde{\str A}$ be the copy of $\str A$ induced on $\{f(r_2),f(r_3), v\}$. Put 
\begin{align*}
s   &= s(\widetilde{\str A})\\
  &= s(v)|_{j(f(r_3))}\\
  &= s(f(r_1))|_{j(f(r_0))}{}\conc q \conc s'
\end{align*}
for some sequence $s'$ and note that $n(\widetilde{\str A}) = n$ (which was defined when we chose $r_2$). By the choice of $r_3$ we know that $\lvert s\rvert > n$ and so $\ell(s,n)$ is defined. By the choice of $n$ we know that $$w(s(f(r_1))|_{j(f(r_0))}) < n$$ and from the choice of $q$ we have that $$w(s(f(r_1))|_{j(f(r_0))}) + q + 1 = n + p,$$ hence $\ell(s,n) = j(f(r_0))+1$. Thus indeed
$$c(\widetilde{\str A}) = w(s|_{j(f(r_0))+1}) - n = p.$$
\end{proof}

\section{Conclusion}\label{sec:conclusion}
Being the first positive big Ramsey result for random structures in infinite languages, this paper helps locate the boundaries of finiteness of big Ramsey degrees. However, there still remain a lot of open problems even regarding unrestricted structures.

While we have proved finiteness of big Ramsey degrees, we do not know their exact values or descriptions. We believe that this is an important problem:
\begin{problem}\label{prob:exact}
Characterise the exact big Ramsey degrees for structures considered in Theorem~\ref{thm:structures}.
\end{problem}

The exact big Ramsey degrees are not known even for the random 3-uniform hypergraph. Below we identify some meaningful intermediate steps towards a full solution to Problem~\ref{prob:exact}.

\begin{problem}\label{prob:exact_steps}
Characterise the exact big Ramsey degrees of the following structures:
\begin{enumerate}
  \item The random 3-uniform hypergraph,
  \item the random $L$-hypergraph where $L$ is finite with no unary relations,
  \item the random $L$-hypergraph where $L$ has no unary relations and finitely many relations of each arity, 
  \item the random $L$-hypergraph where $L$ has countably many unary relations and finitely many relations of each arity.
\end{enumerate}
\end{problem}

We believe that already for the 3-uniform hypergraph one will encounter the most crucial differences compared to when the language is only binary. However, having only one relation which is ternary, one will likely be able to phrase the constructions rather concretely. We expect that abstracting the ideas and coming up with a suitable formalism will be the biggest obstacle when generalising this to finite $L$ with no unaries. We believe that there are fundamental differences between finite $L$ and $L$ with finitely many relations of each arity (cf. Conjecture~\ref{conj:no_enveloping}), it is however possible that these differences do not materialise in the context of Problem~\ref{prob:exact_steps} and the proofs will end up being the same. Adding unaries tends to be a rather technical nuisance and we expect this to be that case in Problem~\ref{prob:exact_steps} as well.

\subsection{Infinite big Ramsey degrees}\label{subsec:inf}
In Section~\ref{sec:inf} we proved that one cannot hope to strengthen our methods to prove big Ramsey results with infinitely many binary relations. We are confident that this argument generalises to higher arities and the following problem has a solution (we believe that the concept of \emph{weak types} from~\cite{typeamalg} leads to such a generalisation).
\begin{problem}
Define a concept of tree-likeness such that all structural embedding are tree-like and prove an analogue of Corollary~\ref{cor:inf} for the cases when there are infinitely many relations of some higher arity. 
\end{problem}

However, even for binary relations the situation is interesting. Corollary~\ref{cor:inf} shows that with respect to tree-like embeddings, the infinite-edge-coloured random graph has infinite big Ramsey degree for a particular triple. On the other hand, it is indivisible by the standard argument (try to embed it into one colour class, if it fails then all realisations of some type are in the other colour class which hence contains a monochromatic copy). We do not know what the big Ramsey degree of an edge of some particular colour is:
\begin{question}
Given the infinite-edge-coloured random graph (see Example~\ref{ex:rado}) and its substructure $\str A$ on two vertices, is the big Ramsey degree of $\str A$ finite? Is it finite with respect to tree-like embeddings?\footnote{\label{footnote:solved}After this paper was submitted, Hubička and Zucker answered this question, proving that the big Ramsey degree of a 2-element substructure of the infinite-edge-coloured random graph is infinite.}
\end{question}

Of course, big Ramsey degrees in general for this structure remain open as well:
\begin{question}\label{q:rado}
Does the infinite-edge-coloured random graph have finite big Ramsey degrees?\footnote{See footnote~\ref{footnote:solved}}
\end{question}
A positive answer to this question would likely require developing new non-tree-like methods for finding oligochromatic copies, which would be a major event for the area. Note that this shows that big Ramsey degrees are much more subtle than small Ramsey degrees, because when we only want to find a finite oligochromatic structure, the whole problem only touches finitely many colours and hence reduces to the finite-language problem.

In fact, it is even more subtle than this: Consider the countable infinite-edge-coloured graph $\str G$ such that edges containing the $i$-th vertex only use the first $i$ colours. This is a countable infinite-edge-coloured graph which is universal for all \emph{finite} infinite-edge-coloured graphs as well as all countable finitely-edge-coloured graphs, but it does not embed the infinite-edge-coloured random graph. At the same time, an application of the Milliken theorem on the tree which branches $n+1$ times on level $n$ shows that this graph has finite big Ramsey degrees, and the Laflamme--Sauer--Vuksanovic arguments~\cite{Laflamme2006} give the exact big Ramsey degrees for this structure (which in turn recovers the exact big Ramsey degrees of all finite-edge-coloured random graphs).

We conjecture that the answers to these questions are negative. In fact, we believe that Theorem~\ref{thm:structures} is tight if there are finitely many unary relations (see the following paragraphs for examples why one has to assume this):

\begin{conjecture}\label{conj:inf}
Let $L$ be a relational language with finitely many unary relations and infinitely many relations of some arity $a\geq 2$. Let $\str H$ be a countable unrestricted $L$-structure realising all relations from $L$. Then there is a finite $L$-structure $\str A$ whose big Ramsey degree in $\str H$ is infinite. Moreover, the number of vertices of $\str A$ only depends on $a$.
\end{conjecture}

\medskip

If one allows infinitely many unary relations, there are structures with yet different behaviour. Consider, for example, the language with infinitely many unary relations and the universal structure where each vertex is in exactly one unary. One can define binary relations on this structure by the pair of unaries on the respective vertices. This structure has infinitely many binary relations but it has finite big Ramsey degrees by Theorem~\ref{thm:structures} (in fact, to prove it it suffices to repeat the proof of Theorem~\ref{thm:main} on $\omega$ which has all big Ramsey degrees equal to one by the Ramsey theorem).

Or consider language $L$ with unary relations $U^1,\ldots$ and binary relations $\rel{}{1},\ldots$, let every vertex have exactly one unary relation, only allow any binary relations between vertices from the same unary relation, and only allow binary relations $\rel{}{1},\ldots,\rel{}{i}$ between vertices from unary $U^i$. In other words, look at the disjoint union of infinitely many edge-coloured random graphs such that the $i$-th of them has all vertices in the unary $U^i$ and has $i$ colours. This structure is unrestricted, but Theorem~\ref{thm:structures} does not capture it because of the infinitely many binary relations. However, for a fixed finite $\str A$, we can restrict ourselves to the substructure induced on the unaries which appear in $\str A$ for which Theorem~\ref{thm:structures} can be applied. Since the substructures on different unaries are disjoint, one can then simply add the remaining unaries back.

However, if one allows binary relations even between vertices with different unary relations, but vertices in $U^i$ can only participate in $\rel{}{1},\ldots,\rel{}{i}$ then the argument from the previous paragraph no longer applies, because we have no guarantee that the oligochromatic copy on finitely many unaries will be generic with respect to the rest, that is, that one will be able to extend it to a full copy. (For example, what might happen is that $U^i$ is present in $\str A$, $U^j$ is not and the monochromatic copy of the restriction will be such that there are no edges between its vertices from $U^i$ and vertices from $U^j$.)

We believe that such structures still have finite big Ramsey degrees and that in order to prove it, one only needs to use a stronger statement of Theorem~\ref{thm:structures} which promises that the oligochromatic copy comes from a vector tree which will make it possible to extend this copy to a copy using all unaries. It seems that a proper formulation of Conjecture~\ref{conj:inf} should speak about the tree of types having a dense set of vertices on which it branches uniformly.

\subsection{Small Ramsey degrees and the partite lemma}
Note that the proof of Proposition~\ref{prop:main} actually gives something stronger than just a finite number of colours: It proves that the colour of every copy only depends on how it embeds into its envelope (this is usually called the \emph{embedding type}). Part of the job when characterising the exact big Ramsey degrees is constructing embeddings which realise as few embedding types as possible.

It was discovered by Hubička~\cite{Hubicka2020CS} that even without knowing the exact big Ramsey degrees, one can often use the big Ramsey upper bound to obtain exact small Ramsey degrees. For unrestricted structures, the argument is as follows: Given a finite (enumerated) structure $\str A$, pick an arbitrary relation $\rel{}{} \in L$ of arity $a\geq 2$ and extend $\str A$ to $\str A'$ adding a vertex $b_v$ for every $v\in A$, putting $b_v < w$ for every $v,w\in A$ and $b_v < b_w$ if and only if $v > w$. Finally, add vertices $c_1, \ldots, c_{a-2}$ which come very first in the enumeration. $\str A$ will be a substructure of $\str A'$ and we will only add relations $(c_1,\ldots,c_{a-2},b_v,v)\in \nbrel{\str A'}{}$ for every $v\in A$.

Note that the lexicographic order on $A$ within $\str A'$ is the same as the enumeration of $\str A$ and that if $\str B$ is a substructure of $\str A$ then $\str B'$ is a substructure of $\str A'$. Moreover, $\str A'$ describes one particular embedding type of $\str A$ (namely the one where first all vertices branch and only then they are coded, and they branch so that the lex-order coincides with the enumeration).

If we only colour this particular embedding type of $\str A'$, the arguments in Proposition~\ref{prop:main} give us a monochromatic subtree. Given any finite $\str B$ which contains $\str A$ as a substructure, this subtree will contain a copy of (the embedding type described by) $\str B'$ in which all copies of (embedding types described by) $\str A'$ will be monochromatic. However, as we noted in the previous paragraphs, every copy of $\str A$ inside this $\str B$ has embedding type described by $\str A'$, hence in fact all copies of $\str A$ inside this $\str B$ will be monochromatic. By compactness we did not need to find the whole monochromatic subtree, just a finite initial segment of it, which contains a unique copy of $\str C'$ for some $\str C$ which hence satisfies $\str C\longrightarrow(\str B)^\str A_{k,1}$. Thus we get, in particular, a new proof of the Abramson--Harrington theorem~\cite{Abramson1978}, or in other words, the Nešetřil--R\"odl theorem without any forbidden substructures~\cite{Nevsetvril1977b}. (This will appear in full detail elsewhere.)

\medskip

Using the Abramson--Harrington theorem one can give a simple proof of the partite lemma of Nešetřil and Rödl~\cite{Nevsetvril1989}: We can consider $\str A$-partite structures (say induced, but the non-induced variant can also be done in this way), as structures with unary marks on vertices which describe the projection to $\str A$. The constraints of $\str A$-partiteness are saying that some particular combination of unaries is forbidden to be in a relation together (i.e. they have a projection to $\str A$ where the relation is not present).

Given an $\str A$-partite structure $\str B$ let $\str B^-$ be its reduct forgetting the unaries. By the Abramson--Harrington theorem there is $\str C^-$ such that $\str C^- \longrightarrow (\str B^-)^{\str A}_2$. Let $\str C$ be the $\str A$-partite structure with vertex set $(C^-\times A)$ where the unaries are given by the projection to $A$ and a tuple $((u_i,x_i))_{i\in n}$ is in a relation $\rel{C}{}$ if and only if $(u_i)_{i\in n}\in \nbrel{\str C^-}{}$ and $(x_i)_{i\in n}\in \rel{A}{}$.

A substructure of $\str C$ is \emph{transversal} if it does not contain any two vertices of the form $(u,x), (u,y)$ for some $u\in C^-$ and $x\neq y\in A$. Note that a substructure of a transversal structure is again transversal. By a similar argument as in the proof of Theorem~\ref{thm:main} we can show that for every colouring of transversal copies of $\str A$ in $\str C$ there is a transversal copy of $\str B$ in which all transversal (and thus actually all) copies of $\str A$ have the same colour, thereby proving the partite lemma. Note that conditions of being $\str A$-partite can be described by a set of forbidden substructures which are covered by a relation. From this point of view, Theorem~\ref{thm:structures} can be considered as a big Ramsey generalisation of the partite lemma.

\section*{Acknowledgments}
We would like to thank Stevo Todor\v cevi\'c for discussions about the infinite-edge-coloured random graph, and Andy Zucker for helpful comments regarding the presentation of our results. We would also like to thank all three anonymous referees for valuable comments which significantly helped improve the paper.

\bibliographystyle{amsplain}



\begin{aicauthors}
\begin{authorinfo}[sam]
  Samuel Braunfeld\\
  Computer Science Institute of Charles University (IÚUK)\\
  Charles University\\
  Ma\-lo\-stransk\'e n\'a\-m\v es\-t\'\i~25\\
  Praha~1, Czech Republic\\
  sbraunfeld\imageat{}iuuk\imagedot{}mff\imagedot{}cuni\imagedot{}cz
\end{authorinfo}

\begin{authorinfo}[david]
  David Chodounsk\'y\\
  Department of Applied Mathematics (KAM)\\
  Charles University\\
  Ma\-lo\-stransk\'e n\'a\-m\v es\-t\'\i~25\\
  Praha~1, Czech Republic\\
  and\\
  Institute of Mathematics of the Czech Academy of Sciences\\
  \v{Z}itn\'a~25\\
  Praha~1, Czech Republic\\
  chodounsky\imageat{}math\imagedot{}cas\imagedot{}cz
\end{authorinfo}

\begin{authorinfo}[noe]
  No\'e de Rancourt\\
  Univ. Lille, CNRS, UMR 8524 - Laboratoire Paul Painlev\'e\\
  F-59000 Lille,  France\\
  nderancour@univ-lille.fr
\end{authorinfo}

\begin{authorinfo}[honza]
  Jan Hubi\v cka\\
  Department of Applied Mathematics (KAM)\\
  Charles University\\
  Ma\-lo\-stransk\'e n\'a\-m\v es\-t\'\i~25\\
  Praha~1, Czech Republic\\
  hubicka\imageat{}kam\imagedot{}mff\imagedot{}cuni\imagedot{}cz
\end{authorinfo}

\begin{authorinfo}[jamal]
  Jamal Kawach\\
  Department of Mathematics\\
  University of Toronto\\
  Toronto, Canada\\
  M5S 2E4\\
  jkawach\imageat{}math\imagedot{}toronto\imagedot{}edu
\end{authorinfo}

\begin{authorinfo}[matej]
  Mat\v ej Kone\v{c}n\'{y}\\
  Department of Applied Mathematics (KAM)\\
  Charles University\\
  Ma\-lo\-stransk\'e n\'a\-m\v es\-t\'\i~25\\
  Praha~1, Czech Republic\\
  and\\
  Institute of Algebra\\
  TU Dresden\\
  Dresden, Germany\\
  matej\imagedot{}konecny\imageat{}tu-dresden\imagedot{}de
\end{authorinfo}

\end{aicauthors}

\end{document}